\documentclass[11pt]{article}
\usepackage{amssymb,amsthm,amsmath}
\usepackage{bm}
\usepackage{mathtools}
\usepackage{tikz}
\usepackage{enumerate}
\usetikzlibrary{arrows.meta,
                decorations.markings,
                quotes}
\numberwithin{equation}{section}
\theoremstyle{plain}
\newtheorem{theorem}[equation]{Theorem}
\newtheorem{lemma}[equation]{Lemma}
\newtheorem{corollary}[equation]{Corollary}
\newtheorem{proposition}[equation]{Proposition}
\theoremstyle{definition}
\newtheorem{definition}[equation]{Definition}

\newtheorem*{notation*}{Notation}
\newtheorem{remark}[equation]{Remark}
\newtheorem*{remark*}{Remark}
\newtheorem{example}[equation]{Example}

\newcommand{\N}[0]{\mathbb{N}}
\newcommand{\Z}[0]{\mathbb{Z}}
\newcommand{\C}[0]{\mathbb{C}}

\newcommand{\R}[0]{\mathbb{R}}
\newcommand{\Chi}[0]{\mathcal{X}}
\newcommand{\PRB}[0]{\mathcal{P}} 
\newcommand{\SUM}[0]{\sum\displaylimits}
\newcommand{\PROD}[0]{\prod\displaylimits}
\newcommand{\INT}[0]{\int\displaylimits}
\newcommand{\D}[0]{\mathrm{d}} 

\newcommand{\EXP}[1]{\mathrm{exp}\left(#1\right)}

\newcommand{\A}[0]{\widetilde{A}} 
\newcommand{\PB}[1]{\widetilde{S_{#1}} / S_{#1}} 
\newcommand{\CS}[0]{\mathcal{D}} 
\newcommand{\HYP}[0]{\mathcal{H}} 
\newcommand{\REG}[1]{\mathcal{R}(#1)} 
\newcommand{\IDEAL}[0]{\mathcal{I}}

\newcommand{\ra}[0]{\rightarrow}

\newcommand{\alc}[0]{\mathcal{A}} 
\newcommand{\typeA}[0]{affine arrangement of type $A_n$}
\newcommand{\Weyl}[0]{Weyl group of type $A_n$}
\newcommand{\affWeyl}[0]{affine Weyl group $\A_n$}

\newcommand{\lev}{\mathrm{lev}}
\newcommand{\ubar}{\underline} 

\tikzset{*->-*/.style={decoration={
  markings,
  mark=at position .5 with {\arrow{>}},
  mark=at position 0 with { \draw [fill] (0,0) circle [radius=2pt];},
  mark=at position 1 with { \draw [fill] (0,0) circle [radius=2pt];}},postaction
={decorate}}}
  
\tikzset{dot/.style={draw,shape=circle,fill=black,scale=.5}}

\tikzset{Bullet/.style={draw,shape=circle,fill=black,scale=.5}}

  \tikzset{*--*/.style={decoration={
  markings,
  mark=at position 0 with { \draw [fill] (0,0) circle [radius=2pt];},
  mark=at position 1 with { \draw [fill] (0,0) circle [radius=2pt];}},postaction
={decorate}}}

\usepackage{hyperref}
\usepackage{xcolor}
\hypersetup{
    colorlinks,
    linkcolor={red!50!black},
    citecolor={blue!50!black},
    urlcolor={blue!80!black}
}
\begin{document}
\title{Ordinal pattern probabilities for symmetric random walks}
\author{Hugh Denoncourt \thanks{Email: \href{mailto:Hugh.Denoncourt@Colorado.edu}{Hugh.Denoncourt@Colorado.edu} - No affiliation}}
\date{}
\maketitle
\begin{abstract}
An ordinal pattern for a finite sequence of real numbers is a permutation that records the relative positions in the sequence. For random walks with steps drawn uniformly from $[-1,1]$, we show an ordinal pattern occurs with probability $\frac{|[1,w]|}{2^n n!}$, where $[1,w]$ is a weak order interval in the affine Weyl group $\A_n$. For random walks with steps drawn from a symmetric Laplace distribution, the probability is $\frac{1}{2^n \PROD_{j=1}^n \lev(\pi)_j}$, where $\lev(\pi)_j$ measures how often $j$ occurs between consecutive values in $\pi$. Permutations whose consecutive values are at most two positions apart in $\pi$ are shown to occur with the same probability for any choice of symmetric continuous step distribution. For random walks with steps from a mean zero normal distribution, ordinal pattern probabilities are determined by a matrix whose $ij$-th entry measures how often $i$ and $j$ are between consecutive values. 
\end{abstract}
\section{Introduction} 
\label{s:introduction}
Let $(a_1,\ldots,a_n) \in \R^n$ be an arbitrary finite sequence of real numbers. A permutation $\pi$ such that $\pi(i) = j$ if $a_i$ is the $j$-th largest position is called the \emph{ordinal pattern for $(a_1,\ldots,a_n)$}. 

For a given sequence $Z_1,Z_2,\ldots$ of continuous random variables, it is natural to ask what the probability is that a given permutation $\pi \in S_n$ occurs as an ordinal pattern in a length $n$ subsequence of outcomes. It is known \cite{bandt-shiha} that for exchangeable random variables, such as those that are independent and identically distributed, this probability is $1/n!$ for all $\pi \in S_n$. By contrast, the distribution on $S_n$ is never uniform for ordinal patterns in positions of a random walk when $n \geq 3$.

Exact probabilities have been calculated for ordinal pattern occurrence in a random walk for a few cases. For $n = 3,4$, Bandt and Shiha \cite{bandt-shiha}, DeFord and Moore \cite{deford-moore}, and Zare \cite{zare} gave values for the case of normally distributed steps of mean zero. For $n = 3,4$, DeFord and Moore \cite{deford-moore} gave piece-wise polynomials for the case of uniform distributions on $[b-1,b]$ for $b \in \left[\frac{1}{2},1\right)$. 

In \cite{elizalde-martinez} and \cite{martinez}, Elizalde and Martinez showed that certain pairs of permutations have the same probability of occurring as an ordinal pattern in a random walk regardless of choice of continuous step distribution. Furthermore, Martinez \cite{martinez} gave a detailed description of regions of steps that generate a given ordinal pattern $\pi$ in terms of a hyperplane arrangement equivalent to the braid arrangement. In this paper, we use hyperplane arrangements and the tools developed in \cite{elizalde-martinez} and \cite{martinez} to find probabilities for ordinal pattern occurrence in certain random walks. 
\subsection{Main results}
\label{ss:results}
Let $X_1, X_2,\ldots$ be independent and identically distributed continuous random variables called \emph{steps} with probability density function $f:\R \ra \R$. Let $Z_i := \sum_{j = 1}^{i - 1} X_j$ be the \emph{positions} of the random walk. For $\pi \in S_{n+1}$, let $\PRB(f,\pi)$ denote the probability that $\pi$ occurs as an ordinal pattern in a length $n + 1$ consecutive subsequence of positions generated by $n$ steps drawn from $f$. All of the main results of the paper are statements about $\PRB(f,\pi)$ for various choices of density function $f$.

Let $\lev(\pi)_j$ denote the number of positions $i$ such that $\pi(i) \leq j < \pi(i+1)$ or such that $\pi(i+1) \leq j < \pi(i)$. In Section \ref{s:laplace}, we use a direct calculation to show that when $f$ is a Laplace distribution, the value of $\PRB(f,\pi)$ is computed from the values of $\lev(\pi)_1,\ldots,\lev(\pi)_n$.
\\ \\
\noindent
\textbf{Theorem \ref{t:laplace}.}
Let $\pi \in S_{n+1}$ and let $f$ be the density function for a Laplace distribution with mean zero. Then
\[
\PRB(f,\pi) = \frac{1}{2^n \prod_{j=1}^n \lev(\pi)_j}.
\] 

We say a permutation is \emph{almost consecutive} if its consecutive values are at most two positions apart in its $1$-line notation. Recall that a function is symmetric if $f(-x) = f(x)$ for all $x \in \R$. In Section \ref{s:universal}, we show that $\PRB(f,\pi)$ does not depend upon the choice of symmetric density function $f$ if $\pi$ is almost consecutive. 
\\ \\
\noindent
\textbf{Theorem \ref{t:acformula}.}
Let $\pi \in S_{n+1}$ be an almost consecutive permutation. Let $f:\R \ra \R$ be a symmetric density function for a continuous probability distribution. Then
\[
\PRB(f,\pi) = \frac{1}{2^n \PROD_{i=1}^n \lev(\pi)_i}.
\]

In Section \ref{s:affineA}, we show that when $f$ is the density function for the uniform distribution on $[-1,1]$, we calculate $\PRB(f,\pi)$ by counting regions of the {\typeA} inside a rational polytope constructed from $\pi$. The counted regions correspond to elements in a weak order interval of the {\affWeyl}.
\\ \\
\noindent
\textbf{Theorem \ref{t:weakordertheorem}.}
Let $f = \frac{1}{2}\Chi_{[-1,1]}$ be the uniform density function on $[-1,1]$. Let $\pi \in S_{n+1}$. Then
\[
\PRB(f,\pi) = \frac{|[1,w]|}{2^n n!},
\]
where $[1,w]$ is a weak order interval of the {\affWeyl}. 
\\

A corollary is that $1 (n + 1)$ or $(n + 1) 1$ appears in the $1$-line notation for $\pi \in S_{n+1}$ if and only if $\PRB(f,\pi) = \frac{1}{2^n n!}$. By contrast, for the Laplace and normal distributions, the other values in the $1$-line notation for $\pi$ typically influence the value of $\PRB(f,\pi)$.

In Section \ref{s:normaldistribution}, we show that when $f$ is the density function for a mean zero normal distribution, we can sometimes compare $\PRB(f,\pi)$ to $\PRB(f,\tau)$. Let $\lev(\pi)_{i,j}$ be the number of positions $k \in [n]$ satisfying $\pi(k) \leq i,j < \pi(k+1)$ or $\pi(k+1) \leq i,j < \pi(k)$.
\\ \\
\noindent
\textbf{Theorem \ref{t:comparegaussian}.}
Let $f:\R \ra \R$ be the density function for the normal distribution with mean zero and any variance. Let $\pi, \tau \in S_{n+1}$. Suppose $\lev(\pi)_{i,j} \leq \lev(\tau)_{i,j}$ for all $(i,j) \in \Phi^+$. Then $\PRB(f,\pi) \geq \PRB(f,\tau)$.
\subsection{The larger context for these results}
\label{ss:greatercontext}
Using ordinal patterns to analyze a time series is sometimes called \emph{ordinal analysis}. Bandt and Pompe \cite{bandt-pompe} introduced permutation entropy, which involves computing ordinal pattern frequencies in a time series. Subsequently, many papers suggested applying ordinal analysis to a variety of applied contexts. The survey \cite{amigo} shows that ordinal pattern frequency often captures qualitative features of a time series. For example, iterated map dynamical systems with deterministic chaotic behavior tend to have forbidden ordinal patterns, whereas white noise does not. 

We can interpret the results of this paper as providing a kind of fingerprint for certain random processes. DeFord and Moore \cite{deford-moore} define KL divergence for the distribution of patterns of length $n$ in $Z$ from those in $X$ by
\[
\mathcal{D}_{\text{KL}_n}(X || Z) = \sum_{\pi \in S_n} P_X(\pi) \log\left(\frac{P_X(\pi)}{P_Z(\pi)}\right),
\]
where $X$ and $Z$ are random variables. Thus, comparisons to walks with steps from symmetric uniform and Laplace densities can be made via this version of KL divergence. 

By Theorem \ref{t:acformula}, there exists a value $p_n$ for the probability that an almost consecutive permutation arises from a length $n$ sequence generated by $n - 1$ steps from a symmetric density function. Thus, for fixed $n$, we have a Bernoulli trial whose ``success'' probability is the same for any random walk whose steps have a symmetric density function. In \cite{denoncourt}, the following values are determined for $p_n$:
\begin{center}
\begin{tabular} { c | c }
$n$ & $p_n$\\
\hline
$2$ & $1$\\
$3$ & $1$\\
$4$ & $2/3$\\
$5$ & $5/12$\\
$6$ & $251/960$\\
$7$ & $463/2880$\\
$8$ & $15281 / 161280$
\end{tabular}
\end{center}
Although the symmetric density function hypothesis is limited in scope, the above values hold for the mean zero Gaussian distribution. Thus, after transforming a sequence generated by a random walk to have mean zero steps, it is reasonable to expect to see values close to those given above on a long enough time scale.
\section{Ordinal pattern preliminaries} 
\label{s:ordinal}
Throughout the paper, random walks have $n$ steps and $n + 1$ positions. The $n$ steps are outcomes of $n$ independent and identically distributed continuous random variables $X_1,\ldots,X_n$. Every tuple $(x_1,\ldots,x_n) \in \R^n$ of steps generates a tuple $(z_1,\ldots,z_{n+1}) \in \R^{n+1}$ of walk positions, where $z_1 = 0$, and $z_i = x_1 + \cdots + x_{i-1}$ for $i > 1$. We say a walk $(z_1,\ldots,z_{n+1})$ has \emph{ordinal pattern $\pi \in S_{n+1}$} if $\pi(i) = j$ whenever $z_i$ is the $j$-th largest position of $(z_1,\ldots,z_{n+1})$. We refer to $(x_1,\ldots,x_n)$ as \emph{step coordinates} for the random walk and the generated tuple $(z_1,\ldots,z_{n+1})$ as \emph{walk coordinates} for the random walk. Ordinal pattern probabilities are calculated as integrals over regions of $\R^n$, but the ordinal patterns themselves are permutations in $S_{n+1}$ derived from walk coordinates in $\R^{n+1}$.  

Following \cite{elizalde-martinez}, we define a map $p:\R^n \setminus Z \ra S_{n+1}$ by
\begin{equation*}
p(x_1,\ldots,x_n) = \pi,
\end{equation*}
where 
\begin{equation*}
\pi(i) = \left|\{j \in \{1,\ldots,n+1\}: z_i \geq z_j\} \right|
\end{equation*}
and $Z$ is the measure zero set of steps $(x_1,\ldots,x_n)$ such that $z_i = z_j$ for some $i \neq j$. 
\begin{definition} 
\label{d:generate}
Let $\pi \in S_{n+1}$. Let $\bm{x} \in \R^n$. We say $\bm{x}$ \emph{generates $\pi$} if $p(\bm{x}) = \pi$. We denote the region of tuples that generate $\pi$ by $D_\pi$. Thus,
\[
D_\pi = \{\bm{x} \in \R^n : p(\bm{x}) = \pi \}.
\]
\end{definition}
We interpret $p$ as mapping a tuple of steps to the ordinal pattern of the generated walk positions. The region $D_\pi$ contains all such tuples for a fixed $\pi \in S_{n+1}$. 

In \cite[Section 2]{elizalde-martinez}, Elizalde and Martinez define an \emph{edge diagram} as a collection of oriented vertical line segments connecting $(i,\pi(i))$ and $(i,\pi(i+1))$. The orientation is downward if $\pi(i) > \pi(i+1)$ and upward otherwise. A \emph{level} is a vertical interval, denoted $\ubar{j}$, whose $y$-coordinates are in $[j,j+1]$. In the edge diagram, the \emph{edge} $e_i$ is a formal sum $e_i = \sum_{j=\pi(i)}^{\pi(i+1) - 1} \ubar{j}$ if $\pi(i) < \pi(i+1)$ and $e_i = -\sum_{j = \pi(i+1)}^{\pi(i) - 1} \ubar{j}$ if $\pi(i) > \pi(i+1)$. In Section \ref{s:laplace}, we introduce a tuple $\lev(\pi)$ that records the number of edges that contain $\ubar{j}$ or its negation. (See figure \ref{f:edgediagram} for an example of an edge diagram and $\lev(\pi)$.) An edge $e_i$ contains $\ubar{j}$ or its negation if $\pi(i) \leq j < \pi(i+1)$ or $\pi(i+1) \leq j < \pi(i)$.

We primarily use the edge diagram as a visual description of $\lev(\pi)$. The edge diagram can be read off from the matrix given in the next definition, as can properties of the region $D_\pi$.
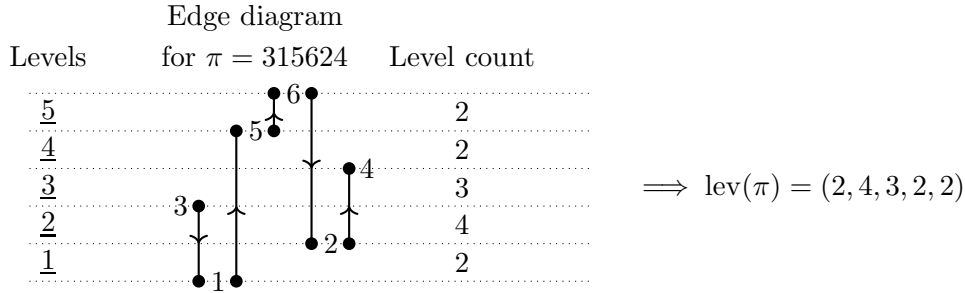
\begin{figure}[htb]
	\centering
		\begin{tikzpicture}[scale=.5]
			\foreach \x in {1,2,...,6}
			\draw[dotted] (0,\x)--(15,\x);

			\draw(20.5,3.5) node{$\implies \lev(\pi) = (2,4,3,2,2)$};

			\draw[thick,*->-*](4.5,3) node[left]{3}--(4.5,1);
			\draw[thick,*->-*](5.5,1) node[left]{1}--(5.5,5);
			\draw[thick,*->-*](6.5,5) node[left]{5}--(6.5,6);
			\draw[thick,*->-*](7.5,6) node[left]{6}--(7.5,2);
			\draw[thick,*->-*](8.5,2) node[left]{2}--(8.5,4) node[right]{4};
			
			\draw(0.5,7) node{Levels};
			\draw(6.0,8) node{Edge diagram};
			\draw(6.0,7) node{for $\pi = 315624$};
			\draw(11.5,7) node{Level count};
			
			{\foreach \x in {1,2,...,5}
			\draw (0.5,0.5+\x) node{$\ubar{\x}$};
			}
			
			\draw (11.5,1.5) node{$2$};
			\draw (11.5,2.5) node{$4$};
			\draw (11.5,3.5) node{$3$};
			\draw (11.5,4.5) node{$2$};
			\draw (11.5,5.5) node{$2$};
		\end{tikzpicture}
		\caption{The edge diagram for a permutation and its level count. Figure adapted from \cite{elizalde-martinez}.}
		\label{f:edgediagram}
\end{figure}
\begin{definition} 
\label{d:represent}
(Martinez \cite[Section 2.1]{martinez}) 
Let $L:S_{n+1} \ra GL_n(\C)$ be defined by $L(\pi) = L_\pi$, where $L_\pi$ is a matrix whose entries are given by
\[
\left(L_\pi\right)_{ij} :=
\begin{cases}
1, & \;\;\; \text{if } \pi(i) \leq j < \pi(i+1),\\
-1, & \;\;\; \text{if } \pi(i+1) \leq j < \pi(i),\\ 
0 & \;\;\; \text{otherwise.}
\end{cases} 
\]
Thus, the $i$-th coordinate of $L_\pi(x_1,\ldots,x_n)$ is given by
\[
L_\pi(x_1,\ldots,x_n)_i = 
\begin{cases} 
x_{\pi(i)} + \cdots + x_{\pi(i+1) - 1}, & \;\; \text{if } \pi(i) < \pi(i+1),\\
-(x_{\pi(i+1)} + \cdots + x_{\pi(i) - 1}), & \;\; \text{if } \pi(i+1) < \pi(i).
\end{cases} 
\]
\end{definition}
Thus, the edges in the edge diagram may be expressed as $e_i = \sum_{j=1}^n (L_\pi)_{i,j} \ubar{j}$. That is, the orientation of edges and which levels appear in edges of an edge diagram can be read from the rows of $L_\pi$. 
\begin{example}
\label{ex:ellpi}
Let $\pi = 315624$. Then
\[
L_\pi = 
\begin{bmatrix}
-1 & -1 & 0 & 0 & 0\\
1 & 1 & 1 & 1 & 0\\
0 & 0 & 0 & 0 & 1\\
0 & -1 & -1 & -1 & -1\\
0 & 1 & 1 & 0 & 0
\end{bmatrix}
.
\]
\end{example}
The next three lemmas capture the basic properties of the representation $L$.
\begin{lemma}
\label{l:represent}
(Martinez \cite[Lemma 2.1.3]{martinez})
The function $L:S_{n+1} \ra GL_n(\C)$ given in Definition \ref{d:represent} is a homomorphism.
\end{lemma}
\begin{lemma}
\label{l:image}
(Martinez \cite[Lemma 2.1.4]{martinez})
Let $\pi \in S_{n+1}$. Then
\[
L_\pi\left(\R_{> 0}^n\right) = D_\pi.
\]
\end{lemma}
\begin{lemma}
\label{l:det}
(Martinez \cite[Lemma 2.1.3]{martinez})
Let $\pi \in S_{n+1}$. Then $\mathrm{det}(L_\pi) = \pm 1$.
\end{lemma}
Since steps are given by independent and identically distributed continuous random variables, the probability that an ordinal pattern $\pi$ occurs depends only on $\pi$ and the associated probability density function $f:\R \ra \R$. The associated joint density function for a walk of $n$ steps is always the function $g:\R^n \ra \R$ defined by $g(x_1,\ldots,x_n) = \prod_{i=1}^n f(x_i)$.
Suppose $f:\R \ra \R$ is a density function. Denote the probability that an ordinal pattern $\pi \in S_{n+1}$ occurs in a random walk of $n$ steps by $\PRB(f,\pi)$, where $f$ is the density function for the steps. Thus,
\begin{equation}
\label{e:density}
\PRB(f,\pi) = \int_{D_\pi} f(x_1) f(x_2) \cdots f(x_n) \; \D x_1 \D x_2 \cdots \D x_n.
\end{equation}
Since $L_\pi$ is invertible and linear, the change-of-variables theorem applies. Since the determinant of $L_\pi$ is $\pm 1$, the absolute value of the Jacobian is always $1$.
\begin{lemma}
\label{l:general}
Let $\pi \in S_{n+1}$. Let $f:\R \ra \R$ be a density function for a continuous distribution. Let $g:\R^n \ra \R$ be the joint density function for the independent steps produced by $f$ defined by $g(x_1,\ldots,x_n) = \prod_{i=1}^n f(x_i)$. Then
\[
\PRB(f,\pi) = \INT_{D_\pi} g (\bm{x}) \; \D \bm{x} = \INT_{\R_{> 0}^n} g(L_\pi(\bm{x})) \; \D \bm{x} = \INT_{\R_{> 0}^n} \prod_{i=1}^n f(L_\pi(\bm{x})_i) \; \D \bm{x}, 
\]
where $L_\pi(\bm{x})_i$ is the $i$-th coordinate of $L_\pi(\bm{x})$.
\end{lemma}
\begin{proof}
The second equality follows from Lemma \ref{l:image}, Lemma \ref{l:det}, and the change-of-variables theorem.
\end{proof}
\begin{example}
\label{ex:generalintegral}
Suppose $f$ is the density function of a continuous distribution. Lemma \ref{l:general} implies
\[
\PRB(f,2413) = \INT_0^\infty \INT_0^\infty \INT_0^\infty f(y + z) \cdot f(-x - y - z) \cdot f(x + y) \; \D x \; \D y \; \D z.
\]
\end{example}
In principle, Lemma \ref{l:general} allows for the calculation of $\PRB(f,\pi)$ for any suitable density function $f$. However, evaluation of the integral is probably computationally infeasible in general. The main reason to use the second integral of Lemma \ref{l:general} instead of the first is that the integration always occurs over $\R_{>0}^n$. 

Since all probability distributions in this paper are assumed to be continuous, the existence of a (not necessarily continuous) density function is guaranteed. Also, walk positions overlap with probability zero, which implies that ordinal pattern probabilities add to $1$. This is also true of many discrete distributions, but we do not address the discrete distribution case in this paper.

Ordinal pattern probabilities are invariant under changes in scale. If $\bm{x} \in D_\pi$, then $c \bm{x} \in D_\pi$ for any $c > 0$. Thus, the region of integration in \eqref{e:density} does not change under the substitution of $c\bm{x}$ for $\bm{x}$, which proves the next lemma. This property is called \emph{scale invariance}.
\begin{lemma}
\label{l:scale-invariance}
Let $f:\R \ra \R$ be a density function of a continuous probability distribution. Let $c \in \R_{>0}$ and let $h:\R \ra \R$ be defined by $h(x) = c f(cx)$. Then $h$ is also a density function and $\PRB(f,\pi) = \PRB(h,\pi)$ for all $\pi \in S_{n+1}$. 
\end{lemma}
\section{Pattern probabilities when steps are from Laplace densities}
\label{s:laplace}
The Laplace distribution, also called the double exponential distribution, has density function $f:\R \ra \R$ given by
\[
f(x) = \frac{1}{2b}\EXP{-\frac{|x - \mu|}{b}},
\]
where $\mu$ is the mean and $b$ is a scale parameter. In this section, we restrict our attention to the mean zero case. A mean zero Laplace distribution arises as the distribution for a random variable expressed as the difference of two identically distributed exponential random variables. 

By Lemma \ref{l:scale-invariance}, ordinal pattern probabilities are scale invariant. Thus, we lose no generality by restricting our attention to the choice $b = 1$. For the remainder of the section, the density function $f$ is defined by 
\[
f(x) = \frac{1}{2}\EXP{-|x|}.
\]
\begin{definition}
\label{d:levelcount}
Let $\pi \in S_{n+1}$. Denote the number of $i \in [n]$ such that $\pi(i) \leq j < \pi(i+1)$ or $\pi(i+1) \leq j < \pi(i)$ by $\lev(\pi)_j$. We call the tuple $(\lev(\pi)_1,\ldots,\lev(\pi)_n)$ the \emph{level count of $\pi$}.  
\end{definition}
Note that $\lev(\pi)_j$ counts the number of times level $\ubar{j}$ is contained in an edge of the edge diagram of $\pi$. (See figure \ref{f:edgediagram}.) Alternativly, it is the sum of the absolute values of the entries in column $j$ of $L_\pi$.
\begin{example}
Let $\pi = 315624$. Then $\lev(\pi) = (2,4,3,2,2)$ as shown in Figure \ref{f:edgediagram}.
\end{example}
Recall from Lemma \ref{l:general} that $\PRB(f,\pi) = \int_{\R_{>0}^n} \prod_{i=1}^n f(L_\pi(\bm{x})_i) \D \bm{x}$.
\begin{theorem}
\label{t:laplace}
Let $\pi \in S_{n+1}$ and let $f$ be the density function for a mean zero Laplace distribution. Then
\[
\PRB(f,\pi) = \frac{1}{2^n \prod_{j=1}^n \lev(\pi)_j}.
\]
\end{theorem}
\begin{proof}
Every factor of the last integrand in Lemma \ref{l:general} has the form $f(\pm(x_a + \cdots + x_b))$, where each $x_k > 0$. Thus,
\[
f(\pm(x_a + \cdots + x_b)) = \EXP{-\left|\pm(x_a + \cdots + x_b)\right|} = \EXP{-(x_a + \cdots + x_b)}.
\]
The term $-x_j$ appears in the above sum whenever $\pi(i) \leq j < \pi(i+1)$ or $\pi(i+1) \leq j < \pi(i)$. 
Thus, by Definition \ref{d:levelcount}, there are $\lev(\pi)_j$ factors of $\prod_{i=1}^n f(L_\pi(\bm{x})_i)$ contributing $-x_j$ inside the exponential for the overall product. By Lemma \ref{l:general},
\begin{align*}
\PRB(f,\pi) &= \INT_{\R_{> 0}^n} \prod_{i=1}^n f(L_\pi(\bm{x})_i) \; \D \bm{x}\\
&= \INT_{\R_{> 0}^n} \frac{1}{2^n}\PROD_{j=1}^n \EXP{-\lev(\pi)_j \; x_j} \; \textrm{d} x_1 \cdots \textrm{d} x_n\\
&= \frac{1}{2^n} \PROD_{j = 1}^n \INT_0^{\infty} \EXP{-\lev(\pi)_j \; x_j} \; \textrm{d} x_j\\
&= \frac{1}{2^n \PROD_{j = 1}^n \lev(\pi)_j}. 
\end{align*}
\end{proof}
\section{Universal pattern probabilities for symmetric step densities} 
\label{s:universal}
Martinez \cite[Section 5]{martinez} introduced a hyperplane arrangement $\mathcal{D}_n$ in $\R^n$ such that for any $\pi \in S_{n+1}$, the set $D_\pi$ is a region of $\mathcal{D}_n$. Furthermore, in \cite[Lemma 5.1.2]{martinez} it was shown that the walls of $D_\pi$ are defined by the row vectors of $L_{\pi^{-1}}$. This allows us to show that for certain $\pi$, we may express $D_\pi$ as a union of cells of the type $B$ Coxeter arrangement, which establishes the main result of this section. Namely, permutations whose consecutive values are at most two positions apart have the same ordinal pattern probabilities as the Laplace distribution. Thus, when $\pi$ is such a permutation, we have $\PRB(f,\pi) = \frac{1}{2^n \PROD_{j=1}^n \lev(\pi)_j}$, \emph{regardless of choice of symmetric density function} for the steps in the random walk. 
\subsection{Hyperplane arrangement preliminaries, notation, and terminology}
\label{ss:hyperplaneterminology}
The hyperplane arrangement notation and terminology we use in this section is similar to that found in  \cite[Section 1.4]{abramenko-brown} or  \cite[Chapter 2]{bjorner-etal}. In particular, a \emph{hyperplane arrangement} is a set $\HYP = \{H_i\}_{i \in I}$ of finitely many hyperplanes. In this section, the arrangements under consideration are \emph{central}, which means they pass through the origin. Thus, associated to each $H_i$ is a linear function $f_i:\R^n \ra \R$ such that $H_i = \{\bm{x} \in \R^n :\;\; f_i(\bm{x}) = 0\}$. For each $H_i \in \HYP$, let
\begin{align*}
H_i^+ &:= \{\bm{x} \in \R^n : \;\;f_i(\bm{x}) > 0\};\\
H_i^0 &:= \{\bm{x} \in \R^n : \;\;f_i(\bm{x}) = 0\}; \text{ and}\\
H_i^- &:= \{\bm{x} \in \R^n : \;\;f_i(\bm{x}) < 0\}.
\end{align*}
A \emph{cell with respect to $\HYP$} is a nonempty set $C$ obtained by choosing for each $i \in I$ a sign $\sigma_i \in \{+,0,-\}$ such that $\bm{x} \in H_i^{\sigma_i}$ for all $\bm{x} \in C$. The sequence $(\sigma_1,\ldots,\sigma_n)$ is called the \emph{sign sequence for $C$}. The cell $C$ is represented by
\begin{equation}
\label{e:halfspaces}
C = \bigcap_{i \in I} H_i^{\sigma_i}.
\end{equation}
The intersection may be redundant. Cells such that $\sigma_i \neq 0$ for all $i \in I$ are called \emph{regions}. The regions of $\HYP$, denoted $\REG{\HYP}$, are the nonempty convex open subsets that partition $\R^n \setminus \cup_{i \in I} H_i$. Note that the collection of all cells partition $\R^n$. However, the cells that are not regions have measure zero and thus contribute nothing to the probability calculations of this section.
\subsection{A hyperplane arrangement for steps of a random walk}
\label{ss:braidarrangement}
Let $H_{i,j}$ be the hyperplane defined by
\[
H_{i,j} = \left\{(x_1,\ldots,x_n) \in \R^n : \;\; x_i + x_{i+1} + \cdots + x_{j-1} = 0\right\}.
\]
Let
\[
\CS_n := \left\{ H_{i,j} : \;\; i,j \in [n] \text{ and } i < j \right\}
\]
be the hyperplane arrangement defined in \cite[Section 5.1]{martinez}. 

As noted in \cite[Section 5.1]{martinez}, the arrangement $\CS_n$ is obtained from the standard braid arrangement via a linear substitution. The arrangements have the same face poset and the same number of regions. However, since the geometry is different and the calculation of $\PRB(f,\pi)$ is not always uniform across regions, we distinguish between the two arrangements in this paper. The next lemma motivates the choice of the arrangement $\CS_n$.
\begin{lemma}
\label{l:regions}
(Martinez \cite[Lemma 5.1.1]{martinez})
The set of regions of $\CS_n$ is $\{D_\pi : \;\;\pi \in S_{n+1}\}$.
\end{lemma}
We say a cell is a \emph{face} of the region $R$ if the cell's sign sequence matches $R$'s sign sequence except for one hyperplane $H$ whose sign is $0$. In this case, we say that $H$ is a \emph{wall} of $R$. For convenience, let $H_{j,i} = H_{i,j}$ when $j > i$. 
\begin{lemma} \label{l:walls} (Martinez \cite[Lemma 5.1.2]{martinez})
Let $\pi \in S_{n+1}$. The set of walls of $D_\pi$ is
\[ 
\left\{ H_{\pi^{-1}(i), \pi^{-1}(i + 1)} \in \CS_n: \;\;i \in [n]\right\}. 
\]
\end{lemma} 
Suppose $\HYP$ and $\HYP'$ be hyperplane arrangements such that $\HYP \subseteq \HYP'$.  Then, the cells of $\HYP$ may be written as a union of cells of $\HYP'$. The next lemma follows from the fact that the collection of walls $\HYP_R$ of a region $R$ forms a hyperplane arrangement in its own right. We use this in Section \ref{ss:typeb} to express $D_\pi$ as a union of cells from the type $B$ Coxeter arrangement.
\begin{lemma}
\label{l:chamberunion}
Let $\HYP_R$ be the collection of walls for a region $R$ of a hyperplane arrangement $\HYP$. If $\HYP'$ is any hyperplane arrangement such that $\HYP_R \subseteq \HYP'$, then
\[
R = \bigcup_{j = 1}^k C_j,
\]
for some collection of cells $C_1, \ldots, C_k$ of $\HYP'$. 
\end{lemma}
\subsection{The type \texorpdfstring{$B$}{B} Coxeter arrangement}
\label{ss:typeb}
We represent a signed permutation on $[n]$ as a pair $(\omega,\epsilon)$, where $\omega \in S_n$ is a permutation on $[n]$ and $\epsilon \in \{-1,+1\}^n$ is a choice of sign for each position. A signed permutation $(\omega,\epsilon)$ acts on $\R^n$ by mapping $(x_1,\ldots,x_n)$ to $\left(\epsilon_1 x_{\omega(1)},\ldots,\epsilon_n x_{\omega(n)}\right)$.
The \emph{type $B$ Coxeter arrangement} is defined by the following hyperplanes:
\begin{align}
\label{e:ijplus}
X_{i,j,+} = \{ (x_1,\ldots,x_n) \in \R^n : x_i + x_j &= 0\};\\
\label{e:ijminus}
X_{i,j,-} = \{ (x_1,\ldots,x_n) \in \R^n : x_i - x_j &= 0\}; \text{ and}\\
\label{e:coords}
X_i = \{ (x_1,\ldots,x_n) \in \R^n : x_i &= 0\}, 
\end{align}
where $i,j \in [n]$ and $i < j$. It is known that the group $B_n$ of all signed permutations acts simply transitively on the regions of the type $B$ hyperplane arrangement, which implies that there are $2^n n!$ regions. See \cite[Section 7]{bjorner-wachs} or \cite[Section 1.15]{humphreys}, for example. Furthermore, the group $B_n$ is generated by reflections so that every group element can be represented as a matrix with determinant $\pm 1$.

For a symmetric density function $f:\R \ra \R$, we have $f(-x) = f(x)$ for all $x \in \R$. Let $g:\R^n \ra \R$ be the joint density function defined by $g(x_1,\ldots,x_n) = \prod_{i=1}^n f(x_i)$. Since $f$ is symmetric and products are invariant under permutations, we have $g(w \cdot \bm{x}) = g(\bm{x})$ for any signed permutation $w \in B_n$. 
\begin{lemma}
\label{l:probregions}
Let $f:\R \ra \R$ be a symmetric density function of a continuous probability distribution. Let $g:\R^n \ra \R$ be the joint density for the random walk of $n$ steps given by $g(x_1,\ldots,x_n) = \prod_{i=1}^n f(x_i)$. Then, for any signed permutation $w \in B_n$, and any region $R$ of the type $B$ hyperplane arrangement, we have
\[
\int_{R} g(\bm{x}) \D \bm{x} = \frac{1}{2^n n!}.
\]
\end{lemma}
\begin{proof}
Let $R_i$ and $R_j$ be arbitrary regions. Since $B_n$ acts simply transitively on regions, there exists $w \in B_n$ such that $w(R_i) = R_j$. Since the absolute value of the Jacobian for $w$ is $1$, the fact that $g(w \cdot \bm{x}) = g(\bm{x})$ and the change-of-variables theorem imply 
\[
\int_{R_i} g(\bm{x}) \D \bm{x} = \int_{R_i} g(w \cdot \bm{x}) \D \bm{x} = \int_{R_j} g(\bm{x}) \D \bm{x}.
\] 
Since there are $2^n n!$ regions, the result follows.
\end{proof}
Recall Lemma \ref{l:chamberunion}: If the walls of a region $R$ lie in an arrangement $\HYP'$ distinct from the one that defined $R$, then we can write $R$ as a union of cells from $\HYP'$. 
Thus, if the walls of a region $D_\pi \in \CS_n$ are type $B$ hyperplanes, the value $\PRB(f,\pi)$ can be calculated by counting type B regions contained in $D_\pi$.
\begin{lemma}
\label{l:countregions}
Suppose $f:\R \ra \R$ is a symmetric density function. Suppose the walls of $D_\pi$ are hyperplanes in the type $B$ Coxeter arrangement. Then
\[
\PRB(f,\pi) = \frac{1}{2^n \PROD_{i=1}^n \lev(\pi)_i}.
\]
\end{lemma}
\begin{proof}
The hypothesis and Lemma \ref{l:chamberunion} imply that
\[
D_\pi = \bigcup_{i = 1}^k R_i \; \cup \; \bigcup_{j=1}^{\ell} C_j,
\]
where each $R_i$ is a region of the type $B$ Coxeter arrangement and each cell $C_j$ is a measure $0$ cell of the arrangement. Thus,
\[
\PRB(f,\pi) = \int_{D_\pi} f(x_1)\cdots f(x_n) \D x_1 \cdots \D x_n = \sum_{i=1}^k \int_{R_i} f(x_1)\cdots f(x_n) \D x_1 \cdots \D x_n.
\]
Since $f$ is symmetric, the joint density function $g(x_1,\ldots,x_n) = \prod_{i=1}^n f(x_i)$ is invariant under the action of $B_n$. 
By Lemma \ref{l:probregions}, the hypotheses imply $\PRB(f,\pi) = \frac{k}{2^n n!}$, where $k$ is the number of type $B$ regions contained in $D_\pi$. Since $k$ depends only on $\pi$, not on the choice of symmetric density function, we may choose $f$ to be the Laplace distribution. The result then follows from Theorem \ref{t:laplace}.
\end{proof}
\subsection{Almost consecutive permutations}
\label{ss:almostconsecutive}
It remains to identify the permutations $\pi \in S_{n+1}$ such that the walls of $D_\pi$ are hyperplanes of the type $B$ Coxeter arrangement. Recall from Lemma \ref{l:walls} that the set of walls for $D_\pi$ is the set of all $H_{\pi^{-1}(i),\pi^{-1}(i+1)}$ such that $i \in [n]$. In Lemma \ref{l:acwalls}, we show the walls of $D_\pi$ are type $B$ hyperplanes if $\pi$ is a permutation whose consecutive values occur no more than two positions apart in its $1$-line notation. These permutations (or their inverses) are called \emph{key permutations} in \cite{page} and \emph{3-determined permutations} in \cite{avgustinovich-kitaev}. In both papers it is shown that the counting sequence for these permutations, which is sequence \href{http://oeis.org/A003274/}{A003274} of the OEIS, grows asymptotically like $(1.4655\ldots)^n$.
\begin{definition}
\label{d:almostconsecutive}
We say $\pi \in S_{n+1}$ is \emph{almost consecutive} if $|\pi^{-1}(i+1) - \pi^{-1}(i)| \leq 2$ for all $i \in [n]$. 
\end{definition}
\begin{example}
Let $\pi = 1423$. Then $\pi$ is almost consecutive since all instances of consecutive values are at most two positions apart in the $1$-line notation. By contrast, the permutation $2413$ is not almost consecutive, since the values $2$ and $3$ are three positions apart.
\end{example}
\begin{lemma}
\label{l:acwalls}
Let $\pi \in S_{n+1}$ be an almost consecutive permutation. Then the walls of $D_\pi$ are hyperplanes of the type $B$ Coxeter arrangement.
\end{lemma}
\begin{proof}
By Lemma \ref{l:walls}, the walls of $D_\pi$ are $H_{\pi^{-1}(i), \pi^{-1}(i+1)}$, where $i \in [n]$. Definition \ref{d:almostconsecutive} then implies that every wall of $D_\pi$ has the form $H_{k,k+1}$ or $H_{k,k+2}$ for some $k$. Thus, a given wall of $D_\pi$ is defined by an equation of the form $x_k = 0$ or $x_k + x_{k+1} = 0$, which is a hyperplane of the form~\eqref{e:coords} or~\eqref{e:ijplus}. In either case, a wall of $D_\pi$ is a hyperplane in the type $B$ Coxeter arrangement.
\end{proof}
\begin{theorem}
\label{t:acformula}
Let $\pi \in S_{n+1}$ be an almost consecutive permutation. Let $f:\R \ra \R$ be a symmetric density function. Then
\[
\PRB(f,\pi) = \frac{1}{2^n \PROD_{i=1}^n \lev(\pi)_i}.
\]
\end{theorem}
\begin{proof}
The result follows from Lemma \ref{l:acwalls} and Lemma \ref{l:countregions}.
\end{proof}
\section{Uniform random walk patterns and Affine \texorpdfstring{$A$}{A}} \label{s:affineA}
Throughout this section, the only density function $f:\R \ra \R$ under consideration is the uniform density function on $[-1,1]$. It is defined by $f(x) = \frac{1}{2}$ for $x \in [-1,1]$ and $f(x) = 0$ otherwise. 

In Section \ref{ss:polytope}, we show that $\PRB(f,\pi)$ is related to the volume of a rational polytope derived from $\pi$. This rational polytope turns out to be an alcoved polytope, which is a union of the regions (called alcoves) of the {\typeA}. A lot is known about the {\typeA} and the {\affWeyl} that acts upon its regions. Thus, the early sections of the chapter are devoted to translating everything into the language of type $A$ root systems. The main result, Theorem \ref{t:weakordertheorem}, states that $\PRB(f,\pi)$ can be computed by counting the number of elements of a weak order interval of $\A_n$. 
\subsection{The polytope of steps that generate \texorpdfstring{$\pi$}{p}}
\label{ss:polytope}
We now define a rational polytope $P_\pi$ that is used to reduce the problem of calculating $\PRB(f,\pi)$ to the problem of calculating the volume of $P_\pi$. 
\begin{definition}
\label{d:rationalpolytope}
Let $\pi \in S_{n+1}$. Let $m_i = \mathrm{min}\{\pi(i),\pi(i+1)\}$ and $M_i = \mathrm{max}\{\pi(i),\pi(i+1)\}$. We call the rational polytope $P_\pi$ satisfying 
\begin{align}
\label{e:rationalcoords}
&x_i \geq 0,\\
\label{e:linearsystem}
0 \leq &x_{m_i} + \cdots + x_{M_i - 1} \leq 1,
\end{align}
for all $i \in [n]$ the \emph{polytope of steps for $\pi$}. 
\end{definition}
\begin{example}
\label{ex:linearsystem}
Let $\pi = 2413$. The system of inequalities defining $P_\pi$ is given by
\begin{align*}
0 \leq x_2 + x_3 &\leq 1,\\
0 \leq x_1 + x_2 + x_3 &\leq 1, \\
0 \leq x_1 + x_2 &\leq 1, \text{ and}\\
x_1,x_2,x_3  &\geq 0.
\end{align*}
\end{example}
Recall that Lemma \ref{l:general} expresses $\PRB(f,\pi)$ as $\int_{\R_{>0}^n} \prod_{i=1}^n f(L_\pi(\bm{x})_i) \D \bm{x}$. Also recall Definition \ref{d:represent}, which expresses the $i$-th coordinate of $L_\pi(x_1,\ldots,x_n)$ as
\[
L_\pi(x_1,\ldots,x_n)_i = 
\begin{cases} 
x_{\pi(i)} + \cdots + x_{\pi(i+1) - 1}, & \;\; \text{if } \pi(i) < \pi(i+1),\\
-(x_{\pi(i+1)} + \cdots + x_{\pi(i) - 1}), & \;\; \text{if } \pi(i+1) < \pi(i).
\end{cases} 
\]
\begin{lemma} 
\label{l:volumepolytope}
Let $f = \frac{1}{2} \Chi_{[-1,1]}$ be the uniform density function on $[-1,1]$. Let $\pi \in S_{n+1}$. Let $P_\pi$ be the polytope of steps for $\pi$. Then
\[
\PRB(f,\pi) = \frac{1}{2^n} \cdot \mathrm{volume}(P_\pi).
\]
\end{lemma}
\begin{proof}
Let $m_i = \mathrm{min}\{\pi(i),\pi(i+1)\}$ and $M_i = \mathrm{max}\{\pi(i),\pi(i+1)\}$.
By Lemma \ref{l:general}, we have
\begin{equation}
\label{e:integralpolytope}
\PRB(f,\pi) = \INT_{\R_{>0}^n} \prod_{i=1}^n \frac{1}{2} \Chi_{[-1,1]}(L_\pi(\bm{x})_i) \; \D \bm{x} = \frac{1}{2^n} \INT_{\R_{>0}^n} \prod_{i=1}^n \Chi_{[-1,1]}(L_\pi(\bm{x})_i) \;\D \bm{x}.
\end{equation}
Let $\bm{x} = (x_1,\ldots,x_n) \in \R_{\geq 0}^n$. Then $x_j \geq 0$ for all $j \in [n]$. Thus $L_\pi(x_1,\ldots,x_n)_i \in [-1,1]$ if and only if
\[
0 \leq x_{m_i} + \cdots + x_{M_i - 1} \leq 1.
\]
The last integrand of \eqref{e:integralpolytope} is $1$ if the system of inequalities defining $P_\pi$ in Definition \ref{d:rationalpolytope} is satisfied, and $0$ otherwise. Thus, the last integral of \eqref{e:integralpolytope} calculates the volume of $P_\pi$.
\end{proof}
\begin{remark}
A consequence of the coordinate inequalities $x_i \geq 0$ and those that have the form $0 \leq x_{m_i} + \cdots + x_{M_i - 1} \leq 1$ is that $x_a + \cdots + x_b \leq 1$ for any $a,b \in [m_i, M_i)$.
In particular, it is a consequence of Lemma \ref{l:existcoordinate} that $x_i \leq 1$ for all $i \in [n]$, which implies $P_\pi \subseteq [0,1]^n$.
\end{remark}
\subsection{Type \texorpdfstring{$A$}{A} root system preliminaries}
Let $\epsilon_1,\ldots,\epsilon_{n+1}$ be the standard basis of $\R^{n+1}$. Let $(\cdot,\cdot)$ be the standard inner product on $\R^{n+1}$. Let
\[ 
V = \left\{\lambda \in \R^{n+1} : \; \;(\lambda, \epsilon_1 + \cdots + \epsilon_{n+1}) = 0 \right\}.
\]
The set
\[
\Phi = \left\{\epsilon_i - \epsilon_j \in V : \; i,j \in [n+1] \text{ and } i < j\right\}
\]
is called the \emph{root system of type $A_n$}. The sets
\begin{align*}
\Phi^+ &= \{\epsilon_i - \epsilon_j \in V : \; i,j \in [n+1] \text{ and } i < j\} \text{ and}\\
\Phi^- &= \{-\lambda \in V : \; \lambda \in \Phi^+\},
\end{align*}
respectively, are called the set of \emph{positive roots} and the set of \emph{negative roots}, respectively. 
\begin{notation*}
We often abbreviate $\epsilon_i - \epsilon_j \in \Phi^+$ by $(i,j) \in \Phi^+$.
\end{notation*}
Let $\alpha_i = \epsilon_i - \epsilon_{i+1}$. Then
\[
\Delta = \{ \epsilon_i - \epsilon_{i+1} \in V: \;\; i \in [n]\} = \{ \alpha_i \in V: \;\; i \in [n]\}
\]
is a basis for $V$. The vectors $\alpha_1,\ldots,\alpha_n$ contained in $\Delta$ are called \emph{simple roots}. There is a dual basis to $\Delta$ consisting of vectors $\omega_1,\ldots,\omega_n$ satisfying $(\omega_i, \alpha_j) = \delta_{ij}$. The dual basis is called the basis of \emph{fundamental coweights}.

The \emph{\Weyl} is the group generated by reflections about the hyperplanes orthogonal to the simple roots. Explicitly, the reflection $s_i$ about the hyperplane orthogonal to $\alpha_i$ is given by
\begin{equation}
\label{e:reflection}
s_i (\lambda) = \lambda - (\lambda, \alpha_i) \alpha_i.
\end{equation}
The map that sends the adjacent transposition $(i \;\; i+1) \in S_{n+1}$ to the reflection $s_i \in A_n$ is called the \emph{geometric representation}. It is a faithful representation of the symmetric group as a Coxeter group. See \cite[Section 4.2]{bjorner-brenti2}, for example.

The representation $L$ given in Definition \ref{d:represent} is closely related to the geometric representation of $S_{n+1}$ as the {\Weyl}.
\begin{lemma}
\label{l:matrixgeometric}
The matrix representation of $s_i$ in the basis of simple roots is $I + M$, where $I$ is the identity matrix, and $M$ is the matrix whose only nonzero entries $A$ are given by $M_{i,i-1} = 1$, $M_{i,i} = -2$, and $M_{i,i+1} = 1$.
\end{lemma}
\begin{proof}
This follows directly from \eqref{e:reflection} and appears in the proof of \cite[Proposition 4.2.1]{bjorner-brenti2}.
\end{proof}
\begin{lemma}
\label{l:coxeterrep}
Let $\pi \in S_{n+1}$. The matrix representation of $\pi$ in the basis of simple roots is $L_\pi^T$. Consequently, the matrix $L_\pi$ is the matrix representation of $\pi$ in the basis of fundamental coweights.
\end{lemma}
\begin{proof}
Recall from Lemma \ref{l:represent} that the function $L:S_{n+1} \ra GL_n(\R)$ that maps $\pi$ to $L_\pi$ is a representation. Thus, it suffices to check the result for the adjacent transpositions. 

Let $\pi$ be the adjacent transposition $(i \;\; i+1)$. We may exhaustively check that $L_\pi^T$ is the geometric representation given in Lemma \ref{l:matrixgeometric}. 

Note that $\pi(j) = j$ and $\pi(j+1) = j + 1$ except for $j \in \{i-1,i,i+1\}$. Thus all rows of $L_\pi$ match the identity matrix except rows $i-1$, $i$, and $i+1$. 

Since $\pi(i-1) = i-1$, and $\pi(i) = i + 1$, the $(i-1)$-st row of $L_\pi$, if it exists, has a $1$ in columns $i-1$ and $i$ and $0$'s in all other positions. Similarly, if row $i+1$ exists, there is a $1$ in columns $i$ and $i+1$ and $0$'s in all other positions. Since $\pi(i) = i+1$ and $\pi(i) = i$, the only nonzero entry of row $i$ is a $-1$ in column $i$.

In summary, we may wite $L_\pi$ as $I + N$, where the only nonzero entries of $N$ are given by $N_{i-1,i} = 1$, $N_{i,i} = -2$, and $N_{i+1,i} = 1$. This is the transpose of the matrix for the geometric representation given in Lemma \ref{l:matrixgeometric}.
\end{proof}
\subsection{The affine arrangement of type \texorpdfstring{$A_n$}{An} in step coordinates}
The definition of the {\typeA} and its connected components involve inner products of the form $(\lambda, \epsilon_i - \epsilon_j)$. Note that $\lambda$, expressed in the basis of fundamental coweights as $x_1 \omega_1 + \cdots + x_n \omega_n$, satisfies
\begin{align*}
(\lambda, \epsilon_i - \epsilon_j) &= (x_1 \omega_1 + \cdots + x_n \omega_n\; , \; \alpha_i + \cdots + \alpha_{j-1})\\
&= x_i + \cdots + x_{j-1}.
\end{align*}
The linear isomorphism mapping $x_1 \omega_1 + \cdots + x_n \omega_n$ to $(x_1,\ldots,x_n)$ translates results about the {\typeA} to results about $P_\pi$. We refer to the image $(x_1,\ldots,x_n) \in \R^n$ of this isomorphism as \emph{step coordinates} in reference to the steps of the random walk. Whenever it makes sense, we expand the standard results and definitions about the {\typeA} into the basis $\omega_1,\ldots,\omega_n$ in anticipation of what is needed to calculate the volume of $P_\pi$.
\begin{definition}
\label{d:consecutiveaffine}
Let $(i,j) \in \Phi^+$ and $a \in \Z$. Let
\begin{align*}
H_{ij}^a &= \{\lambda \in V : (\lambda, \epsilon_i - \epsilon_j) = a\}\\
&= \{x_1 \omega_1 + \cdots + x_n \omega_n \in V : x_i + \cdots + x_{j-1} = a\}.
\end{align*}
The collection of all hyperplanes of the form $H_{ij}^a$ is called the \emph{\typeA}. The connected components of $V \setminus \cup H_{ij}^a$ are called \emph{alcoves}. The group generated by the set of reflections about hyperplanes of the form $H_{ij}^a$ is the \emph{\affWeyl}.
\end{definition}
Let $\alc$ be an alcove of the affine walk arrangement. For any $\lambda \in \alc$, and any pair $(i,j) \in \Phi^+$, Definition \ref{d:consecutiveaffine} implies the existence of an  integer $k_\alc (i,j)$ such that $(\lambda, \epsilon_i - \epsilon_j)$ is strictly between $k_\alc (i,j)$ and $k_\alc (i,j) + 1$. 
\begin{definition}
\label{d:address}
Let $\alc$ be an alcove of the affine walk arrangement. Let $\Phi^+$ be the set of positive roots. The function $k_\alc : \Phi^+ \ra \Z$ such that
\begin{align*}
\alc &= \left\{ \lambda \in V : \; \;
k_\alc (i,j) < (\lambda, \epsilon_i - \epsilon_j) < k_\alc (i,j) + 1 \text{ for all } (i,j) \in \Phi^+ \right\}\\
&= \left\{ x_1 \omega_1 + \cdots + x_n \omega_n  \in V: \;\; k_\alc (i,j) < x_i + \cdots + x_{j-1} < k_\alc (i,j) + 1 \text{ for all } (i,j) \in \Phi^+\right\}
\end{align*}
is called the \emph{address of $\alc$}. The alcove
\begin{align*}
\alc_{\circ} &= \left\{\lambda \in V : \;\; 0 < (\lambda,\epsilon_i - \epsilon_j) < 1 \textrm{ for all } (i,j) \in \Phi^+ \right\}\\
&= \left\{x_1 \omega_1 + \cdots + x_n\omega_n \in V : \;\; 0 < x_i + \cdots + x_{j-1} < 1 \textrm{ for all } (i,j) \in \Phi^+ \right\}
\end{align*}
is called the \emph{fundamental alcove}. 
\end{definition}
Thus, the fundamental alcove is the unique alcove whose address is the constant zero function from $\Phi^+$ to $\Z$. 
\begin{example}
\label{ex:alcove}
Every point in the unit hypercube that is not in the measure zero union of hyperplanes of the affine walk arrangement lies in some alcove. For example, the point $(0.2, 0.05, 0.8, 0.4, 0.7)$ in step coordinates is in the alcove $\alc$ whose address is shown in Figure \ref{f:alcove}.
\end{example}
\begin{figure}[htb] 
	\centering
		\begin{tikzpicture}[scale=.25]
		    \foreach \i in {1,2,3,4,5}   {
		        \pgfmathsetmacro{\a}{int(\i + 1)}
		        \foreach \j in {\a,...,6} {
		            \draw (2*\i,2*\j) node{$\i \j$};
		        }   
		    }
		    \draw (14.5,9.5) node{$k_\alc$};
		    \draw[very thick,->] (13,7) .. controls (14,9.5) and (15,6) .. (17,7);
		    \draw (20, 4) node{$0$};
		    \draw (20, 6) node{$0$};
		    \draw (20, 8) node{$1$};
		    \draw (20, 10) node{$1$};
		    \draw (20, 12) node{$2$};
		    \draw (22, 6) node{$0$};
		    \draw (22, 8) node{$0$};
		    \draw (22, 10) node{$1$};
		    \draw (22, 12) node{$1$};
		    \draw (24, 8) node{$0$};
		    \draw (24, 10) node{$1$};
		    \draw (24, 12) node{$1$};
		    \draw (26, 10) node{$0$};
		    \draw (26, 12) node{$1$};
		    \draw (28, 12) node{$0$};
		\end{tikzpicture}
		\caption{The address of the alcove $\alc$ containing the point in Example \ref{ex:alcove}.}
		\label{f:alcove}
\end{figure}
The group $\A_n$ has generating set $s_1,\ldots,s_{n+1}$, where $s_1,\ldots,s_n$ are the same generators from $A_n$ that reflect about the hyperplanes $H_{i \; i+1}^0$. The generator $s_{n+1}$ reflects about the hyperplane $H_{1  n}^1$. Thus, the action of $s_i$ is to swap the $i$-th and $(i+1)$-st coordinates of elements of $V$. The action of $s_{n+1}$ is to swap the first and last coordinates, add one to the first coordinate and subtract one from the last coordinate. See \cite[page 86]{shi1} or \cite[Section 4.3]{humphreys}, for example.

The first part of the next lemma provides a correspondence between the the group $\A_n$ and the alcoves of the {\typeA}. The last part provides the link to calculating the volume of $P_\pi$. Recall that Lemma \ref{l:coxeterrep} identifies $L_\pi^T$ as the matrix representing $\pi$ in the geometric representation. Also recall from Lemma \ref{l:det} that the determinant of $L_\pi$ is $\pm 1$.
\begin{lemma}
The following are true about the {\affWeyl}.
\label{l:affineprelims} 
\begin{enumerate}[(i)]
\item
The {\affWeyl} acts simply transitively on the alcoves of the {\typeA}. 
\item 
Every element of $\A_n$ is a product of an element of $A_n$ and a translation. 
\item
Elements of $\A_n$ acting on step coordinates are volume-preserving on $\R^n$ relative to the standard inner product on $\R^n$ and Lebesgue measure.  
\end{enumerate}
\end{lemma}
\begin{proof}
Part (i) is \cite[Theorem 4.5]{humphreys}. Part (ii) is \cite[Proposition 4.2]{humphreys}. 

Lemma \ref{l:coxeterrep} shows that elements of $A_n$ expressed as matrices relative to the basis of fundamental coweights have the form $L_\pi$ for some $\pi \in S_{n+1}$. Since translation preserves volume in any basis under any inner product, part (ii) and Lemma \ref{l:det} prove part (iii). 
\end{proof}
\begin{remark}
When we convert to coordinates in $\R^n$ via the basis of fundamental coweights, we are calculating volumes and integrals with a standard Lebesgue measure on $\R^n$ equipped with the standard inner product. This is not the same inner product as the one on $V$. To see this difference in inner product visually, compare \cite[Figure 5.1]{martinez} to a standard centrally-symmetric representation of the braid arrangement in the plane.
\end{remark}
Part (i) of Lemma \ref{l:affineprelims} ensures that $w(\alc_\circ)$ in the next definition is an alcove.
\begin{definition}
\label{d:alcovelabel}
Let $w \in \A_n$. The \emph{alcove of $w$}, denoted $\alc_w$, is the alcove $w(\alc_\circ)$.
\end{definition}
\subsection{Computing the volume of \texorpdfstring{$P_\pi$}{Pp} by counting alcoves}
\begin{lemma} 
\label{l:polytopeunion}
Let $\pi \in S_{n+1}$ and let $P_\pi$ be the polytope of steps for $\pi$. Let $\alc$ be an alcove of the {\typeA} expressed in step coordinates. Then $\alc \subseteq P_\pi$ or $\alc \cap P_\pi = \emptyset$. 
\end{lemma}
\begin{proof}
The address $k_\alc: \Phi^+ \ra \Z$ for $\alc$ determines a system of inequalities where each inequality has the form $k_\alc(i,j) < x_i + \cdots + x_{j-1} < k_\alc(i,j) + 1$, for each $(i,j) \in \Phi^+$. This includes the pairs $(m_i, M_i - 1) \in \Phi^+$ in Definition \ref{d:rationalpolytope}. If $k_\alc(m_i,M_i - 1) = 0$ for all $i \in [n]$, then every $\bm{x} \in \alc$ satisfies all the inequalities that define $P_\pi$, which implies $\alc \subseteq P_\pi$. Otherwise, the sum of coordinates $x_{m_i} + \cdots + x_{M_i - 1}$ is incompatible with $P_\pi$ for some $i \in [n]$, which implies $\alc \cap P_\pi = \emptyset$.
\end{proof}
\begin{lemma}
\label{l:affinealcovecount}
Let
\begin{align*}
P &= \left\{ \lambda \in V : \;\; 0 < (\lambda, \alpha_i) < 1 \text{ for all } i \in [n]\right\}\\
&= \left\{ x_1 \omega_1 + \cdots + x_n \omega_n : \;\; 0 < x_i < 1 \text{ for all } i \in [n]\right\}.
\end{align*}
In step coordinates, the parallelepiped $P$ is the unit cube $[0,1]^n$. There are $n!$ alcoves of the {\typeA} contained in $P$.
\end{lemma}
\begin{proof}
See the proof of \cite[Theorem 4.9]{humphreys} or \cite[Section 3]{lam-postnikov2}.
\end{proof}
\begin{corollary}
\label{c:polytopevolume}
In step coordinates, each alcove of the {\typeA} has volume $1/n!$. Thus,
\[
\PRB(f,\pi) = \frac{K_\pi}{2^n n!},
\]
where $K_\pi$ is the number of alcoves contained in $P_\pi$.
\end{corollary}
\begin{proof}
The set of points not in any alcove has measure zero.  Thus part (iii) of Lemma \ref{l:affineprelims} and Lemma \ref{l:affinealcovecount} show that alcoves have volume $1/n!$ in step coordinates. The result then follows from Lemma \ref{l:polytopeunion} and Lemma \ref{l:volumepolytope}.
\end{proof}
Not every function from $\Phi^+$ to $\Z$ is the address of an alcove. A characterization of such functions is given by Shi's Theorem. See \cite[Lemma 6.1.3]{shi1} or \cite[Theorem 5.2]{shi2}.
\begin{theorem} 
\label{t:shi}
(Shi's Theorem) A function $k:\Phi^+ \ra \Z$ is the address of an alcove if and only if
\[
k(i,t) + k(t,j) \leq k(i,j) \leq k(i,t) + k(t,j) + 1
\]
for all $i,t,j$ satisfying $1 \leq i < t < j \leq n + 1$.
\end{theorem}
Shi's Theorem and Corollary \ref{c:polytopevolume} provide a straightforward, though inefficient, method for computing $\PRB(f,\pi)$. This method, and an alternative one based on \cite{stanley}, is given in \cite{denoncourt}.
\begin{proposition}
\label{p:computeuniform}
Let $\pi \in S_{n+1}$. Let $f$ be the uniform density function on $[-1,1]$. Let $K_\pi$ denote the number of functions $k:\Phi^+ \ra \N$ satisfying the inequalities
\[
k(i,t) + k(t,j) \leq k(i,j) \leq k(i,t) + k(t,j) + 1,
\]
where $i < t < j$, and also satisfying the equalities $k(i,j) = 0$ whenever there exists $c$ such that $\pi(c)  \leq i < j \leq \pi(c+1)$ or $\pi(c+1) \leq i < j \leq \pi(c)$. Then
\[
\PRB(f,\pi) = \frac{K_\pi}{2^n n!}.
\] 
\end{proposition}
\subsection{A characterization of the weak order in terms of alcove addresses}
\label{ss:weakordersection}
Recall that $\A_n$ is generated by reflections $s_1,\ldots,s_{n+1}$. The \emph{length of $w$}, denoted $\ell(w)$, is the smallest number of generators in an expression of $w$ as a product of generators. Define a relation $\ra$ by the condition $w \ra ws$ if $s$ is a generator and $\ell(ws) > \ell(w)$. The \emph{weak order} on $\A_n$ is defined as the transitive closure of the relation $\ra$. 

The main result of this section, Lemma \ref{l:weakcharacter1}, characterizes the weak order on $\A_n$ in terms of alcove addresses. It might be folklore or known. There is an indirect way to prove the lemma by combining \cite[Theorem 4.1]{shi3} with \cite[Theorem 5.3]{bjorner-brenti2}. The approach given below uses a geometric characterization of the weak order on $\A_n$ given in \cite{humphreys}.

For a given hyperplane $H_{ij}^a$ of the {\typeA}, two sides of the hyperplane are determined by the conditions $(\lambda, \epsilon_i - \epsilon_j) > a$ and $(\lambda, \epsilon_i - \epsilon_j) < a$. We say a hyperplane $H$ \emph{separates} $\alc$ from $\alc_\circ$ if $\alc$ and $\alc_\circ$ lie on two sides of $H$. Based on the conditions for determining sides, we determine whether $H_{ij}^a$ separates $\alc$ and $\alc_\circ$ from the the address of $\alc$.
\begin{lemma}
\label{l:separatingaddress}
Let $H_{ij}^a$ be a hyperplane in the {\typeA}, let $\alc_\circ$ denote the fundamental alcove, and let $\alc$ be an arbitrary alcove. If $a > 0$, then $H_{ij}^a$ separates $\alc$ from $\alc_\circ$ if and only if $k_\alc(i,j) \geq a$. If $a \leq 0$, then $H_{ij}^a$ separates $\alc$ from $\alc_\circ$ if and only if $k_\alc(i,j) \leq a - 1$. 
\end{lemma}
\begin{proof}
Suppose $a > 0$. Since $k_{\alc_\circ}(i,j) = 0$, we have $\alc_\circ$ on the side of $H_{ij}^a$ where $(\lambda,\epsilon_i - \epsilon_j) < a$. Note that $\alc$ is on the side where $(\lambda, \epsilon_i - \epsilon_j) > a$ if and only $k_\alc(i,j) \geq a$. Thus $H_{ij}^a$ separates $\alc_\circ$ and $\alc$ if and only if $k_\alc(i,j) \geq a$.

The argument for $a < 0$ is similar. 
\end{proof}
\begin{lemma}
\label{l:separatingweak}
Let $\mathcal{L}(w)$ be the set of hyperplanes separating $\alc_w$ from $\alc_\circ$. Then $u \leq w$ in the weak order if and only if $\mathcal{L}(u) \subseteq \mathcal{L}(w)$.
\end{lemma}
\begin{proof}
This is \cite[Theorem 4.5]{humphreys}.
\end{proof}
Let $k:\Phi^+ \ra \Z$ and $k':\Phi^+ \ra \Z$ be addresses. We write $k' \leq_A k$ if $k'(i,j) \leq k(i,j)$ whenever both are nonnegative or $k'(i,j) \geq k(i,j)$ whenever both are nonpositive. We write $k' \leq k$ if $k'(i,j) \leq k(i,j)$ for all $(i,j) \in \Phi^+$, which is the standard notation for function comparison.
\begin{lemma}
\label{l:weakcharacter1}
Let $u,w \in \A_n$. Then $u \leq w$ if and only if $k_u \leq_A k_w$. 
\end{lemma}
\begin{proof}
The result follows from Lemma \ref{l:separatingaddress} and Lemma \ref{l:separatingweak}.
\end{proof}
The addresses of alcoves in $P_\pi$ are all greater than equal to $0$. Thus, we simplify the previous lemma to characterize weak order as a comparison of addresses as functions.
\begin{corollary}
\label{c:weakcharacter2}
Let $u,w \in \A_n$. Suppose $k_u(i,j) \geq 0$ and $k_w(i,j) \geq 0$ for all $(i,j) \in \Phi^+$. Then $u \leq w$ if and only if $k_u \leq k_w$. 
\end{corollary}
\subsection{Ideals in the root poset determine the alcoves in \texorpdfstring{$P_\pi$}{Pp}}
\label{ss:rootposet}
If we set $k(i,j) = 0$ whenever required by Proposition \ref{p:computeuniform}, and  greedily set $k(i,j)$ to the maximum amount allowed by Shi's theorem, then we obtain a maximal address satisfying the system of linear inequalities defining the polytope $P_\pi$. By Corollary \ref{c:weakcharacter2}, if this turns out to be a unique maximum address satisfying the system, then the alcoves in $P_\pi$ correspond to a weak order interval of $\A_n$. We use a construction due to Sommers \cite{sommers} to show that this is the case. 

There is a standard order $\leq$ on $\Phi^+$, called \emph{the root poset}, such that $(i',j') \leq (i,j)$ if and only if $i \leq i' < j' \leq j$, which is equivalent to $[i',j'] \subseteq [i,j]$. Recall that an ideal is a down-closed subset of a poset. 
\begin{definition} \label{d:ideal}
Let $\pi \in S_{n+1}$. For $i \in [n]$, we say $(\pi(i),\pi(i+1))$ is a \emph{consecutive root for $\pi$} if $\pi(i) < \pi(i+1)$. Similarly, if $\pi(i+1) < \pi(i)$, we say $(\pi(i+1),\pi(i))$ is a \emph{consecutive root for $\pi$}. Denote the collection of consecutive roots for $\pi$ by $C_\pi$. Define the \emph{root ideal of $\pi$}, denoted $\IDEAL_\pi$, by
\begin{equation*}
\IDEAL_\pi := \{(i',j') : (i',j') \leq (i,j) \text{ for some } (i,j) \in C_\pi\}
\end{equation*}
\end{definition}
The motivation for defining $\IDEAL_\pi$ comes from the next lemma, which states that the address of any alcove in $P_\pi$ is $0$ on the ideal $I_\pi$. 
\begin{lemma} \label{l:ideal}
Let $k_\alc:\Phi^+ \ra \Z$ be the address of an alcove $\alc$ in the polytope $P_\pi$ of steps for $\pi$. For any $(i',j') \in \IDEAL_\pi$ and any $(i,j) \in C_\pi$ such that $k(i,j) = 0$, we have $k(i',j') = 0$.
\end{lemma}
\begin{proof}
Given that $(i,j) \in C_\pi$, we know $x_i + \cdots + x_{j-1} \leq 1$. Thus, if $i' > i$ and $j' < j$, we know $x_{i'} + \cdots + x_{j' - 1} \leq 1$. It follows that $k(i',j') = 0$.  
\end{proof}
In the next definition, it is more convenient to regard elements of $\Phi^+$ as vectors, rather than using our abbreviation as pairs $(i,j)$ of integers.
\begin{definition}
\label{d:rootideal}
For a fixed root $\alpha \in \Phi^+$ and a fixed ideal $\IDEAL$ of $\Phi^+$, let $\alpha_{\IDEAL}$ be defined by
\[
\alpha_{\IDEAL} = \text{min} \; \left\{k : \sum_{i = 1}^{k + 1} \gamma_i = \alpha  \text{ with } \gamma_i \in \IDEAL \right\}.
\]
\end{definition}
In other words, the smallest number of joins needed to express $\alpha$ as a join in the root poset using only elements of $\IDEAL$ is $\alpha_{\IDEAL} + 1$. The value of $\alpha_{\IDEAL}$ is zero for any element of $\IDEAL$. 

As in Section \ref{ss:weakordersection}, we write $k' \leq k$ if $k'(i,j) \leq k(i,j)$ for all $(i,j) \in \Phi^+$ for addresses that are always nonnegative. The next lemma is a dual version of \cite[Theorem 2]{armstrong}. 
\begin{lemma} 
\label{l:sommers}
(Sommers \cite[Section 5]{sommers})
For any ideal $\IDEAL$ of $\Phi^+$ that contains all the simple roots, there exists a unique maximum address $k_\IDEAL:\Phi^+ \ra \Z$ such that $k_\IDEAL(i,j) = 0$ for all $(i,j) \in \IDEAL$. It is defined by
\[
k_\IDEAL(i,j) = (i,j)_{\IDEAL}.
\]
\end{lemma}
\begin{proof}
In the proof of \cite[Lemma 5.1 part (2)]{sommers}, it is shown that $k(i,j) \leq (i,j)_\IDEAL$ for any address satisfying $k(i,j) = 0$ for all $(a,b) \in \IDEAL$. In \cite[Lemma 5.2 part (2)]{sommers}, it is shown that there exists an address $k$ such that $k(i,j) = (i,j)_\IDEAL$ for all $(i,j) \in \Phi^+$. Since any address $k'$ satisfying $k'(i,j) = 0$ for all $(i,j) \in \IDEAL$ must also satisfy $k'(i,j) \leq (i,j)_\IDEAL = k(i,j)$, it follows that $k$ is the unique maximum address such that $k(i,j)$ is zero on $\IDEAL$.
\end{proof}
\begin{example}
\label{ex:ideal}
The alcove address of Figure \ref{f:alcove} has $k(i,j) = 0$ for any $(i,j)$ where $j = i + 1$ as well as $(1,3)$ and $(2,4)$. The maximum alcove guaranteed by Lemma \ref{l:sommers} is obtained by filling the entries with the maximum possible value that the conditions of Shi's theorem allows. The address of this alcove is given in Figure \ref{f:maximumalcove}. Its values are, as expected, larger than those of Figure \ref{f:alcove}.
\end{example}
\begin{figure}[htb] 
	\centering
		\begin{tikzpicture}[scale=.25]
		    \foreach \i in {1,2,3,4,5}   {
		        \pgfmathsetmacro{\a}{int(\i + 1)}
		        \foreach \j in {\a,...,6} {
		            \draw (2*\i,2*\j) node{$\i \j$};
		        }   
		    }
		    \draw (14.5,9.5) node{$k_\IDEAL$};
		    \draw[very thick,->] (13,7) .. controls (14,9.5) and (15,6) .. (17,7);
		    \draw (20, 4) node{$0$};
		    \draw (20, 6) node{$0$};
		    \draw (20, 8) node{$1$};
		    \draw (20, 10) node{$1$};
		    \draw (20, 12) node{$2$};
		    \draw (22, 6) node{$0$};
		    \draw (22, 8) node{$0$};
		    \draw (22, 10) node{$1$};
		    \draw (22, 12) node{$2$};
		    \draw (24, 8) node{$0$};
		    \draw (24, 10) node{$1$};
		    \draw (24, 12) node{$2$};
		    \draw (26, 10) node{$0$};
		    \draw (26, 12) node{$1$};
		    \draw (28, 12) node{$0$};
		\end{tikzpicture}
		\caption{The maximum address that is zero on the ideal $\IDEAL$ of Example \ref{ex:ideal}.}
		\label{f:maximumalcove}
\end{figure}
To apply Lemma \ref{l:sommers} to $\IDEAL_\pi$ requires that $\IDEAL_\pi$ contain all the simple roots. 
\begin{lemma}
\label{l:existcoordinate}
Let $\pi \in S_{n+1}$. Let $m_i = \mathrm{min}\{\pi(i),\pi(i+1)\}$ and $M_i = \mathrm{max}\{\pi(i),\pi(i+1)\}$. For any $j \in [n]$, there exists $i \in [n]$ such that $m_i \leq j < M_i$. Thus every $(j,j+1) \in \Delta$ is in $\IDEAL_\pi$. 
\end{lemma}
\begin{proof}
Suppose otherwise. Let $k$ be such that $\pi(k) = j$. Then both $\pi(k-1)$ and $\pi(k+1)$, if defined, must be greater than $j$. If there exists an index $b$ such that $\pi(b) > j$ and $\pi(b+1) < j$, or vice versa, then $i = b$ is such that $m_i \leq j < M_i$. Thus, to the left of $k - 1$ and to the right of $k + 1$, the values must stay above $j$. Since $\pi$ is a permutation, this implies $j = 1$. However, one of $\pi(k-1)$ or $\pi(k+1)$ is defined, and $\pi(k)$ is the minimum of the two values, which implies $m_i \leq j < M_i$ for either $i = k - 1$ or $i = k$. 
\end{proof}
\begin{corollary}
Let $f = \frac{1}{2}\Chi_{[-1,1]}$ be the uniform density function on $[-1,1]$. Let $\tau,\pi \in S_{n+1}$. Suppose $\IDEAL_\pi \subseteq \IDEAL_\tau$. Then
\[
\PRB(f,\pi) \geq \PRB(f,\tau).
\] 
\end{corollary}
\begin{proof}
Lemma \ref{l:existcoordinate} implies that $\IDEAL_\pi$ and $\IDEAL_\tau$ contain all the simple roots. The hypothesis $\IDEAL_\pi \subseteq \IDEAL_\tau$ implies that $k_{\IDEAL_\tau}(i,j) = 0$ whenever $k_{\IDEAL_\pi}(i,j) = 0$. Lemma \ref{l:sommers} implies that $k_{\IDEAL_\tau}(i,j) \leq k_{\IDEAL_\pi}(i,j)$ for all $(i,j) \in \Phi^+$. The result then follows from Corollary \ref{c:weakcharacter2}.
\end{proof}
\begin{theorem}
\label{t:weakordertheorem}
Let $f = \frac{1}{2}\Chi_{[-1,1]}$ be the uniform density function on $[-1,1]$. Let $\pi \in S_{n+1}$ and let $\IDEAL_\pi$ be the root ideal of $\pi$. Let the address of $w \in \A_n$ be given by $k_w(i,j) = (i,j)_{\IDEAL_\pi}$ for all $(i,j) \in \Phi^+$. Then,
\[
\PRB(f,\pi) = \frac{|[1,w]|}{2^n n!},
\]
where $[1,w]$ consists of all $v \in \A_n$ such that $v \leq w$ in the weak order on $\A_n$.
\end{theorem}
\begin{proof}
Lemma \ref{l:existcoordinate} implies that we may apply Lemma \ref{l:sommers} to $\IDEAL_\pi$. The result then follows from Lemma \ref{l:sommers} and Corollary \ref{c:weakcharacter2}.
\end{proof}
Weak order intervals of $\A_n$ satisfy condition $(1)$ of \cite[Proposition 3.5]{lam-postnikov2}, which implies $P_\pi$ is an \emph{alcoved polytope} in the sense of \cite{lam-postnikov1} and \cite{lam-postnikov2}. Thus \cite[Theorem 3.2]{lam-postnikov1} provides yet another computational approach to calculating the volume of $P_\pi$, although we do not pursue that approach in this paper. 

The next proposition is somewhat surprising, in the sense that two consecutive entries of $\pi$ can completely determine $\PRB(f,\pi)$. This is not the case for the Laplace or normal density functions, and it is reasonable to suspect that a typical density function does not exhibit this property. 
\begin{proposition}
\label{p:minimumprob}
Let $f = \frac{1}{2}\Chi_{[-1,1]}$ be the uniform density function on $[-1,1]$. Let $\pi \in S_{n+1}$. Then $1 (n + 1)$ or $(n + 1) 1$ occur in consecutive positions in the $1$-line notation for $\pi$ if and only if
\[
\PRB(f,\pi) = \frac{1}{2^n n!}
\]
\end{proposition}
\begin{proof}
If $1$ and $n + 1$ are consecutive in the $1$-line notation, then $\IDEAL_\pi$ is all of $\Phi^+$. Thus $w = 1$ in Theorem \ref{t:weakordertheorem}, which implies there is only one element of $\A_n$ in the interval $[1,w]$.

Conversely, if there are no consecutive occurrences of $1$ and $(n+1)$, then the ideal $\IDEAL_\pi$ does not contain $(1,n)$. Thus $(1,n)$ is the join of at least $2$ elements of $\IDEAL_\pi$, which implies $k_{\IDEAL_\pi}(1,n) > 0$. This implies $[1,w]$ contains more than one element.
\end{proof}
\section{Pattern probability comparisons for the normal distribution}
\label{s:normaldistribution}
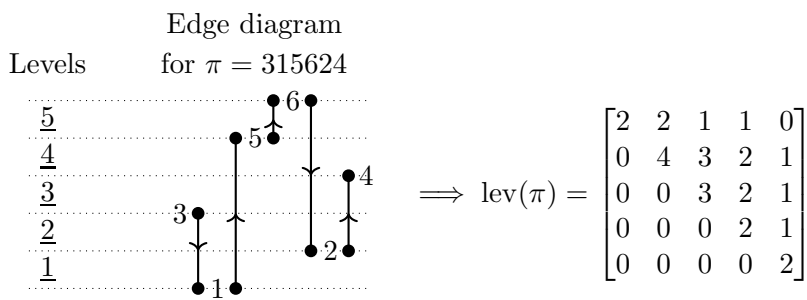
\begin{figure}[htb] \label{f:edgediagram2}
	\centering
		\begin{tikzpicture}[scale=.5]
			\foreach \x in {1,2,...,6}
			\draw[dotted] (0,\x)--(9,\x);

			\draw(15.5,3.5) node{$\implies \lev(\pi) = \begin{bmatrix} 2 & 2 & 1 & 1 & 0\\ 0 & 4 & 3 & 2 & 1\\ 0 & 0 & 3 & 2 & 1\\ 0 & 0 & 0 & 2 & 1\\0 & 0 & 0 & 0 & 2 \end{bmatrix}$};

			\draw[thick,*->-*](4.5,3) node[left]{3}--(4.5,1);
			\draw[thick,*->-*](5.5,1) node[left]{1}--(5.5,5);
			\draw[thick,*->-*](6.5,5) node[left]{5}--(6.5,6);
			\draw[thick,*->-*](7.5,6) node[left]{6}--(7.5,2);
			\draw[thick,*->-*](8.5,2) node[left]{2}--(8.5,4) node[right]{4};
			
			\draw(0.5,7) node{Levels};
			\draw(6.0,8) node{Edge diagram};
			\draw(6.0,7) node{for $\pi = 315624$};
			
			{\foreach \x in {1,2,...,5}
			\draw (0.5,0.5+\x) node{$\ubar{\x}$};
			}
		\end{tikzpicture}
		\caption{The edge diagram for a permutation and its level count matrix. Figure adapted from \cite{elizalde-martinez}.}
\end{figure}
As Zare \cite{zare} suggested, when $f$ is a normal distribution, we calculate $\PRB(f,\pi)$ by finding the volume of a spherical simplex. General equations exist to compute such volumes. See \cite{aomoto} or \cite{ribando}, for example. However, they appear to be computationally intensive, as is Lemma \ref{l:general} when it is applied to the normal distribution. Nonetheless, there are a few direct comparisons we can make involving alcoves and levels of the edge diagram.

Recall that the alcoves of Section \ref{s:affineA} are simplices of volume $1/n!$, by Corollary \ref{c:polytopevolume}. For a given origin-centered ball $B$ in $\R^n$, we obtain an underestimate for $\PRB(f,\pi)$ by counting all alcoves in $D_\pi$ that are fully contained in $B$. Similarly, we obtain an overestimate for $\PRB(f,\pi)$ by counting all alcoves in $D_\pi$ that intersect $B$ or are fully contained in $B$. The address of an alcove and the radius of the ball suffice to determine whether an alcove is fully contained in $B$ or intersects $B$ or is disjoint from $B$. 
\begin{proposition} 
\label{p:bounds}
Let $B$ be an origin-centered ball in $\R^n$. Let $m_\pi$ be the number of alcoves fully contained in $D_\pi$ and $B$. Let $M_\pi$ be the number of alcoves fully contained in $D_\pi$ that have nonempty intersection with $B$. Then
\[
\frac{m_\pi}{n! \text{ volume}(B)} \leq \PRB(f,\pi) \leq \frac{M_\pi}{n! \text{ volume}(B)}.
\]
\end{proposition}
Note that hypercubes with integer-valued vertices could be used instead of alcoves, but one would need to determine whether the hypercube is fully contained in $D_\pi$ or intersects $D_\pi$ or is disjoint from $D_\pi$. For alcoves, this is directly determined from the alcove's address.

For $\pi \in S_{n+1}$, we defined $\lev(\pi)$ on $[n]$ to measure how often a value lies between two consecutive values of $\pi$. We extend the definition of $\lev(\pi)$ to arbitrary pairs of $[n]$.
\begin{definition}
\label{d:levelmatrix}
Let $\pi \in S_{n+1}$. Denote the number of positions $k$ such that $\pi(k) \leq i,j < \pi(k+1)$ or $\pi(k+1) \leq i,j < \pi(k)$ by $\lev(\pi)_{i,j}$. Note that $\lev(\pi)_{i,i}$ is the same as $\lev(\pi)_i$ defined in Definition \ref{d:levelcount}.
\end{definition}
The measure of the spherical simplex that determines $\PRB(f,\pi)$ for the normal distribution is completely determined by the values of $\lev(\pi)$, as will be seen in the proof of Theorem \ref{t:comparegaussian}. Although such measures may be difficult to calculate, we can sometimes use $\lev(\pi)$ and $\lev(\tau)$ to compare $\PRB(f,\pi)$ and $\PRB(f,\tau)$. 

Recall that Lemma \ref{l:general} expresses $\PRB(f,\pi)$ as $\int_{D_\pi} g(L_\pi(\bm{x})) \D \bm{x}$, where $g$ is the joint density function defined by $g(x_1,\ldots,x_n) = f(x_1)f(x_2)\cdots f(x_n)$.
\begin{theorem} 
\label{t:comparegaussian}
Let $f:\R \ra \R$ be the density function for a normal distribution with mean zero and any variance. Let $\pi, \tau \in S_{n+1}$ and suppose $\lev(\pi)_{i,j} \leq \lev(\tau)_{i,j}$ for all $(i,j) \in \Phi^+$. Then $\PRB(f,\pi) \geq \PRB(f,\tau)$.
\end{theorem}
\begin{proof}
By scale invariance (Lemma \ref{l:scale-invariance}), we may assume $f$ is given by $f(x) = Ke^{-x^2}$ for some $K \in \R$. Every factor in Lemma \ref{l:general} has the form $f(\pm(x_a + \cdots + x_b))$. We have
\begin{align*}
f(\pm(x_a + \cdots + x_b)) &= \EXP{-(x_a + \cdots + x_b)^2}\\
                          &= \EXP{-\SUM_{k=a}^b x_k^2} \; \cdot \; \EXP{-\sum 2 x_k x_{k'}},
\end{align*}
where the sum in the second exponential is over all pairs between $a$ and $b$ (inclusive). By Definition \ref{d:levelmatrix}, there are $\lev(\pi)_{j,j}$ factors contributing one term of the form $-x_j^2$ and $\lev(\pi)_{i,j}$ factors contributing one term of the form $-2x_i x_j$ to the overall product of exponentials. Thus, if $\lev(\pi)_{i,j} \leq \lev(\tau)_{i,j}$ for all $(i,j)$, the integrand in Lemma \ref{l:general} for $\pi$ is always at least as large as the integrand for $\tau$. 
\end{proof}
In \cite[Lemma 2.3]{elizalde-martinez}, Elnitsky and Martinez showed that if $L_\pi$ can be obtained from $L_\tau$ by a permutation of rows and columns, then $\PRB(f,\pi) = \PRB(f,\tau)$ for any choice of density function $f$, symmetric or otherwise. By including their guaranteed equalities, we obtain more comparable pairs of permutations than what is guaranteed by Theorem \ref{t:comparegaussian}. 

Write $\pi \equiv \tau$ if $L_\pi$ can be obtained from $L_\tau$ by a permutation of rows and columns. Write $\pi \leq_{\mathrm{lev}} \tau$ if $\lev(\pi)_{i,j} \leq \lev(\tau)_{i,j}$ for all $i,j \in [n]$. Define $\leqslant$ as the transitive closure of of the relation $\equiv$ and the partial order $\leq_{\mathrm{lev}}$. We then have a broader collection of comparable pairs of permutations for the normal distribution.
\begin{corollary} 
\label{c:comparegaussian}
Let $f$ be a normal distribution with mean zero. Let $\pi, \tau \in S_{n+1}$. If $\pi \leqslant \tau$, then $\PRB(f,\pi) \geq \PRB(f,\tau)$.
\end{corollary}
\section{Concluding remarks and problems} 
\label{s:postscript}
In \cite{zare}, Zare asks which permutations occur most frequently in random walks with a normal or uniform distribution of mean zero for its steps. Our results provide an imprecise heuristic: permutations with large consecutive changes in its $1$-line notation are less likely to occur than permutations with small consecutive changes. In other words, for permutations where $\lev(\pi)_{i,j}$ is large, we expect $\PRB(f,\pi)$ to be small, and vice versa, for a large class of symmetric density functions of a continuous probability distribution.

As a general problem, we would like to know what general hypotheses are needed to prove $\PRB(f,\pi) \geq \PRB(f,\tau)$ whenever $\lev(\pi)_{i,j} \leq \lev(\tau)_{i,j}$ for all $(i,j) \in \Phi^+$. However, this question is probably too open-ended. We have evidence for the following more precise conjecture.
\\ \\
\textbf{Conjecture:} Let $f:\R \ra \R$ be a density function that is log-concave on $\R_{>0}$ and symmetric on $\R$. Let $\pi,\tau \in S_{n+1}$ and suppose $\lev(\pi)_{i,j} \leq \lev(\tau)_{i,j}$ for all $(i,j) \in \Phi^+$. Then $\PRB(f,\pi) \geq \PRB(f,\tau)$.
\\ \\
From the perspective of computation, Proposition \ref{p:computeuniform} provides a direct approach to computing ordinal pattern probabilities when the steps are uniform. Theorem \ref{t:weakordertheorem} reduces the problem to finding the size of a weak order interval in $\A_n$. A lot is known about these intervals, which is enough to make the computation easier in some cases. For example, in \cite{lapointe-morse}, Lapointe and Morse show that the weak order on the quotient $\PB{k+1}$ is order-isomorphic to the $k$-Young lattice. Furthermore, some intervals of the $k$-Young lattice are intervals of the Young lattice. The size of intervals of the Young lattice is given by a classical determinant formula due to Kreweras, thus providing an alternative calculation to Proposition \ref{p:computeuniform} for some permutations. (See \cite[Section 2.3.7]{kreweras}.)

However, the affine symmetric group contains many weak order intervals isomorphic to weak order intervals of the symmetric group. By \cite[Theorem 1.4]{dittmer-pak}, computing the size of weak order intervals in $S_n$ is $\#P$-complete. Unless there is something special about the weak order intervals in Theorem \ref{t:weakordertheorem}, computing $\PRB(f,\pi)$ is hard when $f$ is uniform. 
\\ \\
\textbf{Conjecture}: Computing $\PRB(f,\pi)$ for the uniform density function $f$ on $[-1,1]$ and arbitrary $\pi \in S_n$ is $\#P$-complete.
\begin{center}
ACKNOWLEDGEMENTS
\end{center}
The author would like to thank Jim and Kate Daly and Emily Pavey for their support. The author also thanks Dana Ernst, Michael Falk, and Jim Swift of NAU's Algebra, Combinatorics, Geometry, and Topology Seminar for their comments on a talk based on an earlier version of this paper. 

\end{document}